\documentclass{mcom-l}

\theoremstyle{definition}

\theoremstyle{remark}

\usepackage{graphicx}
\usepackage{latexsym}
\usepackage{amsmath}
\usepackage{amssymb}
\usepackage{amsfonts}
\usepackage{verbatim}
\usepackage{mathrsfs}
\usepackage[modulo]{lineno} 
\usepackage[latin1]{inputenc}
\usepackage{color}
\usepackage[colorlinks,citecolor=blue,urlcolor=blue]{hyperref}

\newtheorem{tm}{Theorem}[section]
\newtheorem{rk}{Remark}[section]
\newtheorem{ap}{Assumption}[section]

\newtheorem{lm}{Lemma}[section]

\newcommand{\cc}{\mathbb C}

\newcommand{\pp}{\mathbb P}
\newcommand{\nn}{\mathbb N}
\newcommand{\rr}{\mathbb R}

\newcommand{\ttt}{\mathbb T}
\newcommand{\zz}{\mathbb Z}

\newcommand{\LL}{\mathcal L}

\newcommand{\FFF}{\mathscr F}

\newcommand{\<}{\langle}
\renewcommand{\>}{\rangle}
\allowdisplaybreaks \allowdisplaybreaks[4]

\begin{document}

\title[Approximation of Stochastic Semiclassical Schr\"odinger Equation]
{Strong Approximation of Stochastic Semiclassical Schr\"odinger Equation with Multiplicative Noise}

\author{Lihai JI} 
\address{Institute of Applied Physics and Computational Mathematics, Beijing 100094, China, and Shanghai Zhangjiang Institute of Mathematics, Shanghai 201203, China}
\email{jilihai@lsec.cc.ac.cn}

\author{Zhihui LIU}
\address{Department of Mathematics \& National Center for Applied Mathematics Shenzhen (NCAMS) \& Shenzhen International Center for Mathematics, Southern University of Science and Technology, Shenzhen 518055, China}
\email{liuzh3@sustech.edu.cn} 
\thanks{The first author is supported by the National Natural Science Foundation of China (NNSFC), No. 12171047. 
The second author is supported by NNSFC, No. 12101296, Basic and Applied Basic Research Foundation of Guangdong Province, No. 2024A1515012348, and Shenzhen Basic Research Special Project (Natural Science Foundation) Basic Research (General Project), No. JCYJ20220530112814033.} 






\begin{abstract}
We consider the stochastic nonlinear Schr\"odinger equation driven by a multiplicative noise in a semiclassical regime, where the Plank constant is small. In this regime, the solution of the equation exhibits high-frequency oscillations. We design an efficient numerical method combining the spectral Galerkin approximation and the midpoint scheme. This accurately approximates the solution, or at least of the associated physical observables. Furthermore, the strong convergence rate for the proposed scheme is derived, which explicitly depends on the Planck constant. This conclusion implies the semiclassical regime's admissible meshing strategies for obtaining ``correct" physical observables.
\end{abstract}

\maketitle


\section{Introduction}

Many problems of solid state physics require the solution of the following semiclassical nonlinear Schr\"odinger (NLS) equation:
\begin{align}\label{det_sNLS}  
&{ \bf i} \varepsilon\partial_t u_t^{\varepsilon}+\frac{\varepsilon^2}{2}\Delta u_t^{\varepsilon}+F(u_t^\varepsilon)=0, \quad t>0,\quad x\in\mathbb{R}^d, 
\end{align}
where $0<\varepsilon\ll1$ is the scaled Planck constant describing the microscopic and macroscopic scale ratio. Here the solution $\{u_t^{\varepsilon}: t \ge 0\}$ is the electron wave function. It is well known that Eq. \eqref{det_sNLS} propagates oscillations with frequency $1/\varepsilon$ in space and time. In addition, these oscillations pose a huge challenge for the numerical computation. Recently, much progress has been made in this area, such as the time splitting spectral method \cite{BJM2002, BJM2003, JZ2013}, the Gaussian beam method \cite{JWY2011, JY2010}, the Hagedorn wave packet approach \cite{FGL2009, Zhou2014}, and the Frozen Gaussian beam method \cite{Kay2006, LZ2018}. Besides, some efficient, conservative methods are particularly designed for Eq. \eqref{det_sNLS}, see, e.g., \cite{CL22, CZ23}. We refer to the comprehensive review \cite{JMS2011, LL2020} and references therein for more recent studies.

In Eq. \eqref{det_sNLS}, no exterior influence is considered. When it comes to quantum dynamics and studying wave propagation in random media, noise has to be introduced. For instance, in \cite{CMZ2020, JLRZ2011, WL2024}, the semiclassical NLS equations with random initial data and potentials are investigated. In this work, we aim to deepen the understanding of the highly oscillatory behavior of a semiclassical NLS equation with a Wiener process perturbation. To our knowledge, this is the first result of investigating the stochastic semiclassical NLS equation.

Consider the following stochastic semiclassical NLS equation with multiplicative noise on the $d$-dimensional torus $\ttt:=[0,1]^d$:
\begin{align}\label{semi-NLS}  
 &{ \bf i} \varepsilon du_t^{\varepsilon}+(\frac{\varepsilon^2}{2}\Delta u_t^{\varepsilon}+F(u_t^\varepsilon)) dt=G(u_t^\varepsilon)dW_t, \quad (t, x) \in(0,T] \times \ttt, 
\end{align}
 with periodic boundary condition and initial datum $u_0^{\varepsilon}$.
Here $W$ is a $L^2(\ttt;\mathbb{R})$-valued $Q$-Wiener process defined on a complete filtered probability space (see Section \ref{sec2}).
Note that in the case of $\varepsilon\equiv1$, Eq. \eqref{semi-NLS} becomes the classical stochastic NLS equation, which has been studied both theoretically and numerically, see, e.g., \cite{BD2004, CHJ2017, CHLZ2017, CHLZ2019, CLZ2023, CH2016, CHL2017, CHP2016, HWZ2017, HW19, Liu2013} and references therein. 

Designing efficient numerical methods that accurately approximate the solution of Eq. \eqref{semi-NLS}, or at least of the associated physical observables, is a formidable mathematical challenge. This paper investigates the spectral Galerkin approximation and its full discretization via a midpoint scheme for Eq. \eqref{semi-NLS}. It is well known that for the deterministic case, i.e., Eq. \eqref{det_sNLS}, \cite{MPP1999} proved that, for the best combination of the spatial and temporal discretizations, one needs the spatial mesh size $O(\varepsilon)$ and the time step size $O(\varepsilon)$ to guarantee good approximations to all (smooth) observables for sufficiently small $\varepsilon$. Failure to satisfy these conditions leads to wrong numerical observables. This paper attempts to understand the spectral Galerkin approximation's resolution capacity and mesh strategies for the stochastic semiclassical NLS equation \eqref{semi-NLS}. By analyzing the strong convergence rates for the proposed numerical approximation, we obtain the admissible meshing strategies for obtaining ``correct" physical observables in the semiclassical regime.

The rest of the paper is organized as follows.
Section \ref{sec2} introduces proper notations, assumptions, and preliminaries. In Section \ref{sec3}, we give the well-posedness and regularity estimates for the solution of Eq. \eqref{semi-NLS}.
Rigorous strong error estimations for the aforementioned spectral Galerkin approximation and its temporal midpoint full discretization are performed in Sections \ref{sec4} and \ref{sec5}, respectively.

\section{Preliminaries}
\label{sec2}

We first introduce some frequently used notations. 
Let $T\in \rr^*_+=(0,\infty)$ be a fixed terminal time.
For an integer $M\in \nn_+=\{1,2,\cdots\}$, we denote by $\zz_M:=\{0,1,\cdots,M\}$ and $\zz_M^*=\{1,\cdots,M\}$. 
Throughout the paper, we use the notations $L_x^2:=L^2(\ttt;\cc)$, $L_t^p:=L^p([0, T];\cc)$, and $L_\omega^p:=L^p(\Omega;\cc)$, with $p\in [2,\infty]$, to denote the usual Lebesgue spaces over $\cc$, and $H:=L^2(\ttt;\rr)$.
Denote by $A: {\rm Dom}(A)\subset H\rightarrow H$ the periodic Laplacian operator on $H$ with an eigensystem $\{(\lambda_k, e_k)\}_{k=1}^\infty$ where the sequence of eigenvalues $\{\lambda_k\}_{k=1}^\infty$ is listed in an increasing order. 
Let $(-A)^\theta$ be the fractional powers of $-A$, and $\dot H^\theta$ be the domain of $(-A)^{\theta/2}$, with $\theta\in \rr_+$.     
We use $(\LL(H; \dot H^\theta), \|\cdot\|_{\LL(H; \dot H^\theta)})$ to denote the space of bounded linear operators from $H$ to $\dot H^\theta$ for $\theta\in \rr_+$ and $\LL(H):=\LL(H; H)$. 
For convenience, sometimes we use the temporally, sample path, and spatially mixed norm $\|\cdot\|_{L_\omega^p L_t^r \dot H^\theta}$ in different orders, such as
\begin{align*}
\|X\|_{L_\omega^p L_t^r \dot H^\theta}
:=(\int_\Omega (\int_0^T \|X(t,\omega)\|_\theta^r {\rm d}t)^\frac pr \pp({\rm d}\omega))^\frac 1p
\end{align*}
for $X\in L_\omega^p L_t^r \dot H^\theta$, 
with the usual modification for $r=\infty$ or $q=\infty$. 

Then, we introduce the driving stochastic process to study the stochastic semiclassical NLS equation \eqref{semi-NLS}. Let $(\Omega,\FFF,\pp)$ be a probability space with a normal filtration 
$\{\FFF_t\}_{0\leq t\leq T}$.
Let ${\bf Q}\in \LL(H)$ be a self-adjoint and nonnegative definite operator on $H$.
Denote by $U_0:={\bf Q}^{1/2} (H)$ and 
$(\LL_2^\theta:=HS(U_0; \dot H^\theta), \|\cdot\|_{\LL_2^\theta})$ the space of Hilbert--Schmidt operators from $U_0$ to $\dot H^\theta$ for $\theta\in \rr_+$. 
Let $\{W(t):t\in [0,T]\}$ be an $H$-valued ${\bf Q}$-Wiener process in the stochastic basis $(\Omega,\FFF, \{\FFF_t\}_{0\leq t\leq T},\pp)$, i.e., there exists an orthonormal basis $\{g_k\}_{k=1}^\infty$ of $H$ which forms the eigenvectors of $\bf Q$ subject to the eigenvalues $\{q_k\}_{k=1}^\infty$ and a sequence of mutually independent Brownian motions $\{\beta_k\}_{k=1}^\infty $ such that (see \cite[Proposition 2.1.10]{LR15})
\begin{align*}
W(t)
=\sum_{k\in \nn_+} {\bf Q}^{1/2} g_k \beta_k(t)
=\sum_{k\in \nn_+} \sqrt{q_k} g_k \beta_k(t),\quad
0\leq t \leq T.
\end{align*}

We assume that the initial datum $u_0^\epsilon$ is an $\FFF_0$-measurable random variable. 
Our main assumptions on the coefficients of Eq. \eqref{semi-NLS} are the following Lipschitz continuity and linear growth conditions.

\begin{ap} \label{ap-fg}
There exist positive constants $\mu$, $L_1$ and $L_2$ such that 
\begin{align}
\|F(u)-F(v)\| \le L_1 \|u-v\|, \quad & \|F(z)\|_\mu \le L_1 (1+\|z\|_\mu), \label{ap-f}\\[1.5mm]
\|G(u)-G(v)\|_{\LL_2^0} \le L_2 \|u-v\|, \quad & \|G(z)\|_{\LL_2^\mu} \le L_2 (1+\|z\|_\mu),  \label{ap-g}
\end{align}
for all $u,v \in H$ and $z \in \dot H^\mu$.
\end{ap}

\begin{rk}
Let $F$ and $G$ denote the Nymiskii operators associated to some functions $f: \ttt \times \cc \to \cc$ and $g: \cc \to \cc$, respectively:
\begin{align*}
F(u)(x):&=f(x, u(x)),  \quad u\in L_x^2,\ x\in \ttt, \\[1.5mm]
G(u) v(x):&=g(u(x)) v(x),  \quad u\in L_x^2,\ v\in U_0,\ x\in \ttt.
\end{align*}  
A sufficient condition such that the condition \eqref{ap-f} holds with $\mu=1$ is that $f$ possesses bounded partial derivatives. 
Indeed, suppose that $|\partial_1 f(x,y)| \le L_1^f$ and $|\partial_2 f(x,y)| \le L_2^f$ for some constants $L_1^f,L^f_2 \ge 0$, then for any $u,v \in H$ and $z \in \dot H^1$,
\begin{align*}
\|F(u)-F(v)\|^2 =\int_{\ttt} |f(x,u(x))-f(x, v(x))|^2 dx \le (L_2^f)^2 \|u-v\|^2, 
\end{align*}
and
\begin{align*}
\|F(z)\|_1^2&=\|F(z)\|^2+\|D F(z)\|^2 \\[1.5mm]
& \le (\|F(z)-F(0)\| +\|F(0)\|)^2
+ \int_{\ttt} |\partial_1 f(x,z(x))+\partial_2 f(x, z(x)) z'(x)|^2 dx \\[1.5mm]
& \le C (1+L^f_1)^2 + 2 (L_2^f)^2 \|z\|_1^2,
\end{align*}
which imply the condition \eqref{ap-f} with $\mu=1$ and certain $L_1>0$ depending on $L^f_1$ and $L^f_2$.
Here and in the rest of the paper, we use $C$ to denote a generic positive constant independent of the discrete parameters that would differ in each appearance. 

A sufficient condition such that the condition \eqref{ap-g} holds with $\mu=1$ is that $g$ is continuously differentiable with bounded derivative and $g_k\in W^{1,\infty}$ (i.e., $g_k$ and its derivative are essentially bounded on $\ttt$) for each $k\in \nn_+$ such that $\sum_{k\in \nn_+} q_k \|g_k\|^2_{W^{1,\infty}}<\infty$ (see \cite[Example 5.1]{LQ21}).
In fact, for any $u,v\in H$ and $z\in \dot H^1$, we have
\begin{align*} 
\|G(u)-G(v)\|^2_{\LL_2^0} 
&=\sum_{k\in \nn_+} q_k \|(g(u)-g(v)) g_k\|^2
\le \|g'\|_\infty^2 \sum_{k\in \nn_+} q_k \|g_k\|_{L_x^\infty}^2 \|u-v\|^2, 
\end{align*}
and
\begin{align*} 
\|G(z) \|^2_{\LL_2^1} 
&\le C \sum_{k\in \nn_+} q_k 
(\|g(z) g_k\|^2+\|g'(z) z' g_k+g(z) g'_k\|^2)\\[1.5mm] 
&\le C(1+\|g'\|_\infty^2) \sum_{k\in \nn_+} q_k 
\|g_k\|_{W_x^{1,\infty}}^2(1+\|z\|_1)^2,
\end{align*}
which show the condition \eqref{ap-g} with $\mu=1$ and certain $L_2>0$ depending on $\|g'\|_\infty$ and $\sum_{k\in \nn_+} q_k \|g_k\|^2_{W^{1,\infty}}$.
\end{rk}

\begin{rk}
In the case of additive noise, e.g., for $G={\rm Id}$, the identity operator on $L_x^2$, the condition \eqref{ap-g} is equivalent to 
\begin{align}\label{con-add}
L_2:=\|(-\Delta)^{\mu/2} Q^{1/2}\|_{HS}<\infty. 
\end{align} 
\end{rk}

\section{Well-posedness and Moment's Estimates}
\label{sec3}

Under Assumption \ref{ap-fg}, as the coefficients of Eq. \eqref{semi-NLS} is Lipschitz continuous in $H$ and grow linearly in $\dot H^\mu$, we have seen that its $L_\omega^p \dot H^\mu$-well-posedness is straightforward (see \cite[Theorem 2.3]{HL19}).
More precisely, Eq. \eqref{semi-NLS} has a unique mild solution given by
\begin{align}\label{mild}
u_t^\varepsilon=S_t^\varepsilon u^\varepsilon_0+{\bf i} \varepsilon^{-1}\int_0^tS_{t-r}^\varepsilon F(u_r^\varepsilon)dr-{\bf i} \varepsilon^{-1}\int_0^tS_{t-r}^\varepsilon G(u_r^\varepsilon)dW_r,
\end{align}
where $S_t^\varepsilon:=\exp(\frac{{ \bf i}\varepsilon}{2}t\Delta)$, $t \in [0, T]$, denotes the semigroup generated by ${\bf i}\varepsilon\Delta/2$.

In the following, we will present the moment's estimates of the exact solution to Eq. \eqref{semi-NLS} or Eq. \eqref{mild}, which would be valuable in the derivation of error estimations between the exact solution of Eq. \eqref{semi-NLS} and the numerical ones in Sections \ref{sec4} and \ref{sec5}.
If there is a damping term, this moment's estimate could be time-uniform.

\begin{tm} \label{tm-well}
Let $p \ge 2$, $\mu>0$, $u_0^\epsilon \in L^p(\Omega; \dot H^\mu)$, and Assumption \ref{ap-fg} hold.
Then Eq. \eqref{semi-NLS} possesses a unique solution $\{u(t):t \ge 0\}$ and there exists a positive constant $C$ such that  
\begin{align}\label{reg}
\sup_{t \in [0, T]} \|u_t^\varepsilon\|_{L_\omega^p \dot H^\mu} 
\le C e^{CT \varepsilon^{-2}} (1+ \|u^\varepsilon_0\|_{L_\omega^p \dot H^\mu}),
\end{align}  
and for any $t,s \in [0, T]$,
\begin{align}\label{hol}
\|u_t^\varepsilon-u_s^\varepsilon\|_{L_\omega^p L_x^2}
\le C e^{CT \varepsilon^{-2}} (1+ \|u^\varepsilon_0\|_{L_\omega^p \dot H^\mu}) |t-s|^{1/2}.
\end{align}  
In particular if $F(u)={ \bf i} \alpha u$ with $\alpha>0$ and $G={\rm Id}$ such that \eqref{con-add} holds, then there exists a positive constant $C$ such that for any $t\in[0,T]$, 
\begin{align} \label{reg+} 
\|u_t^{\varepsilon}\|_{L_\omega^p \dot H^\mu}
& \le e^{-\alpha t \varepsilon^{-1}} \|u^\varepsilon_0\|_{L_\omega^p \dot H^\mu}
+C (\alpha\varepsilon)^{-1/2}, 
\end{align} 
and for any $t,s \in [0, T]$,
\begin{align}\label{hol+}
\|u_t^\varepsilon-u_s^\varepsilon\|_{L_\omega^p L_x^2}
\le C(e^{-\alpha t \varepsilon^{-1}} \|u^\varepsilon_0\|_{L_\omega^p \dot H^\mu}
+ (\alpha\varepsilon)^{-1/2}) |t-s|^{1/2}.
\end{align}  
\end{tm}

\begin{proof} 
By the Minkowski inequality and the Burkholder--Davis--Gundy inequality, we have
\begin{equation*}
\begin{split}
\|u_t^{\varepsilon}\|_{L_\omega^p \dot H^\mu}
& \le \|S_t^\varepsilon u^\varepsilon_0\|_{L_\omega^p \dot H^\mu}
+ \varepsilon^{-1} \int_0^t \|S_{t-r}^{\varepsilon} F(u_r^{\varepsilon})\|_{L_\omega^p \dot H^\mu}dr \\[1.5mm]
& \quad +  \varepsilon^{-1} (\int_0^t \|S_{t-r}^{\varepsilon} G(u_r^{\varepsilon})\|^2_{L_\omega^p \dot H^\mu}dr )^{1/2}. 
\end{split}
\end{equation*}
Using the equality \eqref{s} in Appendix with $\mu\equiv0$ and Assumption \ref{ap-fg}, we obtain
\begin{align*}
\|u_t^{\varepsilon}\|_{L_\omega^p \dot H^\mu}
& \le \|u^\varepsilon_0\|_{L_\omega^p \dot H^\mu}
+ L_1 \varepsilon^{-1} \int_0^t (1+\|u_r^{\varepsilon}\|_{L_\omega^p \dot H^\mu}) dr\\[1.5mm]
& \quad + L_2 \varepsilon^{-1} (\int_0^t (1+\|u_r^{\varepsilon}\|_{L_\omega^p \dot H^\mu})^2 dr )^{1/2}, 
\end{align*} 
which shows that  
\begin{align*}
\|u_t^{\varepsilon}\|^2_{L_\omega^p \dot H^\mu}
& \le 3 \|u^\varepsilon_0\|^2_{L_\omega^p \dot H^\mu}
+ 3 \varepsilon^{-2}(L^2_1  t+L_2^2) \int_0^t (1+\|u_r^{\varepsilon}\|_{L_\omega^p \dot H^\mu})^2 dr  \\[1.5mm]
& \le C T \varepsilon^{-2} \|u^\varepsilon_0\|^2_{L_\omega^p \dot H^\mu}
+ C T \varepsilon^{-2} \int_0^t \|u_r^{\varepsilon}\|_{L_\omega^p \dot H^\mu}^2 dr.
\end{align*} 
We conclude \eqref{reg} by the Gronwall inequality.

To show the H\"older continuous of the solution $u_t^\varepsilon$, we suppose that $0 \le s<t \le T$ without loss of generality.
According to the mild formulation \eqref{mild} and the Minkowski inequality, we get 
\begin{align*}
\|u_t^\varepsilon-u_s^\varepsilon\|_{L_\omega^p L_x^2} \le I+II+III,
\end{align*}
where
\begin{align*}
I   &:=\|(S^\varepsilon_{t-s}-{\rm Id})u^\varepsilon_s\|_{L_\omega^p L_x^2}, \\[1.5mm]
II  &:=\varepsilon^{-1} \|\int_s^t S^\varepsilon_{t-r} F(u^\varepsilon_r)dr \|_{L_\omega^p L_x^2}, \\[1.5mm]
III &:=\varepsilon^{-1} \|\int_s^t S^\varepsilon_{t-r} G(u^\varepsilon_r)dW(r) \|_{L_\omega^p L_x^2}.
\end{align*}
The smoothing property \eqref{s-con} in Appendix and the estimation \eqref{reg} yield that 
\begin{align} \label{I}
I & \le C \varepsilon^{\mu/2} t^{\mu/2} e^{C t \varepsilon^{-2}} (1+ \|u^\varepsilon_0\|_{L_\omega^p \dot H^\mu})(t-s)^{\mu/2}.
\end{align}

For the second term $II$, by the Minkowski inequality, the isometry property \eqref{s} with $\alpha=0$, the condition \eqref{ap-f}, and the estimation \eqref{reg}, we have
\begin{align} \label{II} 
II & \le \int_s^t \|F(u_r)\|_{L_\omega^p L_x^2} dr \nonumber \\
&\le L_1 \int_s^t (1+\|u_r\|_{L_\omega^p \dot H^\mu}) dr \nonumber  \\
& \le C e^{C T \varepsilon^{-2}} (1+ \|u^\varepsilon_0\|_{L_\omega^p \dot H^\mu})(t-s).
\end{align} 
For the last term $III$, by the Burkholder--Davis--Gundy inequality, the isometry property \eqref{s} with $\alpha=0$, the condition \eqref{ap-g}, and the estimation \eqref{reg}, 
\begin{align} \label{III}
III &\le (\int_s^t \|G(u_r)\|^2_{L_\omega^p \LL_2^0} {\rm d}r )^\frac12 
\le C e^{C T \varepsilon^{-2}}  (1+ \|u^\varepsilon_0\|_{L_\omega^p \dot H^\mu}) (t-s)^\frac12.
\end{align}
Combining the above estimations \eqref{I}, \eqref{II}, and \eqref{III}, we conclude \eqref{hol}. 

Now assume that $F(u)={ \bf i} \alpha u$ and $G={\rm Id}$. 
Due to \eqref{s} and the condition \eqref{con-add}, we obtain
\begin{align*}
\|u_t^{\varepsilon}\|_{L_\omega^p \dot H^\mu}
& \le e^{-\alpha t \varepsilon^{-1}} \|u^\varepsilon_0\|_{L_\omega^p \dot H^\mu}
+ C \varepsilon^{-1} (\int_0^t e^{-2 \alpha r/\varepsilon}\|S_r\|^2_{\LL_2^\mu} dr)^{1/2} \\[1.5mm]
& = e^{-\alpha t \varepsilon^{-1}} \|u^\varepsilon_0\|_{L_\omega^p \dot H^\mu}
+ C L_2 \varepsilon^{-1} (\int_0^t e^{-2 \alpha r/\varepsilon} dr)^{1/2} \\[1.5mm]
& = e^{-\alpha t \varepsilon^{-1}} \|u^\varepsilon_0\|_{L_\omega^p \dot H^\mu}
+ C (\alpha\varepsilon)^{-1/2}.
\end{align*}  
This shows \eqref{reg+}, and we can conclude \eqref{hol+} by the previous argument. The proof is complete.
\end{proof}

\section{Spatial Semi-discretization}
\label{sec4}

In the previous section, we derive the uniform boundedness and the H\"older continuous of the solution to Eq. \eqref{semi-NLS}. 
In this section, we apply the spectral Galerkin method to discretize Eq. \eqref{semi-NLS} spatially to obtain a valid approximation. 

In the sequel, let $N\in \nn_+$ and define the finite dimensional subspace $V_N:={\rm span}\{e_1,e_2,\cdots,e_N\}$.
Then the corresponding numerical solution $\{u^N:=u^{\varepsilon,N}\}$, $N\in\mathbb{N}$, of Eq. \eqref{semi-NLS} satisfies
\begin{align}\label{spe}
\begin{split}
&{ \bf i} \varepsilon du_t^{\varepsilon,N}+(\frac{\varepsilon^2}{2}\Delta u_t^{\varepsilon,N}+P_NF(u_t^{\varepsilon,N}))dt=P_N G(u_t^{\varepsilon,N})dW_t, \quad t \in (0, T], \\[1.5mm]
&u^{\varepsilon,N}_0=P_Nu^\varepsilon_0.
\end{split}
\end{align}
Similarly to Eq. \eqref{mild}, we have 
\begin{align}\label{mild-N}
u_t^{\varepsilon,N}=S_t^N u^\varepsilon_0
+{\bf i} \varepsilon^{-1}\int_0^t S_{t-r}^{\varepsilon,N} F(u_r^{\varepsilon,N})dr-{\bf i} \varepsilon^{-1}\int_0^t S_{t-r}^{\varepsilon,N} G(u_r^{\varepsilon,N})dW_r,
\end{align}
where $S_t^N:=P_NS_t^\varepsilon=\exp(\frac{{ \bf i}\varepsilon}{2}t\Delta_N)$, $t \in [0, T]$.

Denote by $e^N_t:=u_t^\varepsilon-u_t^{\varepsilon,N}$, $t \in [0, T]$.
Then subtracting Eq. \eqref{mild-N} from Eq. \eqref{mild} yields that
\begin{align*}
\begin{split}
e^N_t=&(S_t^\varepsilon-S_t^N) u^\varepsilon_0
+{\bf i} \varepsilon^{-1}\int_0^t [S_{t-r}^\varepsilon F(u_r^\varepsilon)-S_{t-r}^{\varepsilon}P_NF(u_r^{\varepsilon,N})] dr\\[1mm]
&-{\bf i} \varepsilon^{-1}\int_0^t [S_{t-r}^\varepsilon G(u_r^\varepsilon)-S_{t-r}^{\varepsilon}P_N G(u^{\varepsilon,N}_r)] dW_r\\[1.5mm]
&:=I^N+II^N+III^N.
\end{split}
\end{align*} 
For any $t \in [0, T]$ and $\alpha>0$, we define the following three operators:
\begin{align*} 
S_t^{\alpha,\varepsilon}:=&\exp((\frac{{ \bf i}\varepsilon}{2}\Delta-\frac\alpha \varepsilon {\rm Id}) t)=e^{-\alpha /\varepsilon t} S_t^\varepsilon,\\ 
S_t^{\alpha,\varepsilon, N}:=&P_N S_t^{\alpha,\varepsilon}=\exp((\frac{{ \bf i}\varepsilon}{2}\Delta_N-\frac\alpha \varepsilon {\rm Id}) t), \\
e^{\alpha, N}_t:= & u_t^{\alpha,\varepsilon}-u_t^{\alpha,\varepsilon,N}. 
\end{align*}  

Using the semigroup theory, we can derive the following error estimate between the spectral Galerkin approximation \eqref{spe} and the exact solution of Eq. \eqref{semi-NLS}.


\begin{tm} \label{tm-en}
Let $p \ge 2$, $\mu>0$, $u_0^\epsilon \in L^p(\Omega; \dot H^\mu)$, and Assumption \ref{ap-fg} hold.
Then there exists a positive constant $C$ such that
\begin{align} \label{en}
&\sup_{0\le t\le T} \|e^N_t\|_{L_\omega^p L_x^2}
\le C e^{CT \varepsilon^{-2}}
(1+ \|u^\varepsilon_0\|_{L_\omega^p \dot H^\mu})N^{-\mu/d}. 
\end{align} 	
Particularly, if $F(u)={ \bf i} \alpha u$ with $\alpha>0$ and $G={\rm Id}$ such that \eqref{con-add} holds, then   
\begin{align} \label{en+} 
\|e^N_t\|_{L_\omega^p L_x^2}
\le (e^{-\alpha t \varepsilon^{-1}} \|u^\varepsilon_0\|_{L_\omega^p \dot H^\mu}
+C (\alpha \varepsilon)^{-1/2}) N^{-\mu/d}.
\end{align} 	
\end{tm}

\begin{proof}
For the first term $I^N$, we use the inequality \eqref{s-sn} in Appendix to get 
\begin{align*} 
\|(S_t^\varepsilon-S_t^{\varepsilon, N}) u^\varepsilon_0\|^2 
\le \lambda_N^{-\mu}\|u^\varepsilon_0\|_\mu^2,
\end{align*} 
which implies
\begin{align}\label{term_I} 
\|I^N\|_{L_\omega^p L_x^2}
 \le C \lambda_N^{-\mu/2} \|u^\varepsilon_0\|_{L_\omega^p \dot H^\mu}. 
\end{align}

For the second term $II^N$, by the Minkovski inequality, the estimate \eqref{s-sn}, the boundedness of $S_t^\varepsilon P_N$, and the condition \eqref{ap-f}, we have
\begin{equation} \label{term_II}
\begin{split}
\|II^N\|_{L_\omega^p L_x^2}
&\le \varepsilon^{-1}\int_0^t\|(S_{t-r}^\varepsilon-S_{t-r}^{\varepsilon, N})F(u_r^\varepsilon)\|_{L_\omega^p L_x^2}dr\\[1.5mm]
& \quad +\varepsilon^{-1}\int_0^t\|S_{t-r}^{\varepsilon, N}(F(u_r^\varepsilon)-F(u_r^{\varepsilon,N}))\|_{L_\omega^p L_x^2}dr\\[1.5mm]
& \le C\varepsilon^{-1} \lambda_N^{-\mu/2} \int_0^t\|F(u_r^\varepsilon)\|_{L^p(\Omega; \dot{H}^\mu)} dr\\[1.5mm]
& \quad + C \varepsilon^{-1}\int_0^t\|F(u_r^\varepsilon)-F(u_r^{\varepsilon,N})\|_{L_\omega^p L_x^2}dr\\[1.5mm]
& \le C\varepsilon^{-1} \lambda_N^{-\mu/2} (1+\|u^\varepsilon_0\|_{L_\omega^p \dot H^\mu})+C \varepsilon^{-1}\int_0^t\|e^N_r\|_{L_\omega^p L_x^2}dr. 
\end{split}
\end{equation}

For the last term $III^N$, it follows from the Minkovski inequality, the Burkholder--Davis--Gundy inequality, and the condition \eqref{ap-f} that
\begin{equation}\label{term_III} 
\begin{split}
\|III^N\|_{L_\omega^p L_x^2}
& \le \varepsilon^{-1}\left\|\int_0^t(S_{t-r}^\varepsilon-S_{t-r}^{\varepsilon}P_N)G(u_r^\varepsilon) dW_r\right\|_{L_\omega^p L_x^2} \\[1.5mm]
& \quad +\varepsilon^{-1}\left\|\int_0^tS_{t-r}^\varepsilon P_N G(u_r^\varepsilon)-G(u_r^{\varepsilon,N}))dW_r\right\|_{L_\omega^p L_x^2}\\[1.5mm]
&\le C \varepsilon^{-1} \lambda_N^{-\mu/2} (\int_0^t (1+\|u_r^\varepsilon\|^2_{L^p(\Omega; \dot{H}^\mu)}) dr )^{1/2}\\[1.5mm]
& \quad +C \varepsilon^{-1} (\int_0^t\|e^N_r\|^2_{L_\omega^p L_x^2} dr )^{1/2}\\[1.5mm]
&\le C \varepsilon^{-1} \lambda_N^{-\mu/2} (1+\|u^\varepsilon_0\|_{L_\omega^p \dot H^\mu}) +C \varepsilon^{-1} (\int_0^t \|e^N_r\|^2_{L_\omega^p L_x^2} dr )^{1/2}.
\end{split}
\end{equation}
Combing \eqref{term_I}, \eqref{term_II}, and \eqref{term_III}, we obtain
\begin{align*}
\|e^N_t\|_{L_\omega^p L_x^2}
& \le C \varepsilon^{-1} \lambda_N^{-\mu/2} (1+\|u^\varepsilon_0\|_{L_\omega^p \dot H^\mu}) 
+C \varepsilon^{-1}\int_0^t\|e^N_r\|_{L_\omega^p L_x^2}dr\\[1.5mm] 
& \quad +C \varepsilon^{-1} (\int_0^t \|e^N_r\|^2_{L_\omega^p L_x^2} dr )^{1/2},
\end{align*}
which implies that 
\begin{align*}
\begin{split}
\|e^N_t\|^2_{L_\omega^p L_x^2} 
& \le C\varepsilon^{-2} \lambda_N^{-\mu}(1+\|u^\varepsilon_0\|_{L_\omega^p \dot H^\mu}^2)+ C \varepsilon^{-2} (1+t)\int_0^t\|e^N_r\|^2_{L_\omega^p L_x^2}dr.
\end{split}
\end{align*}
Then, the Gronwall inequality, combined with the Weyl law, yields the assertion \eqref{en}.

Now assume that $F(u)={ \bf i} \alpha u$ and $G={\rm Id}$. Then
\begin{align*}
e^N_t=&(S_t^{\varepsilon, \alpha}-S_t^{\varepsilon, \alpha, N}) u^\varepsilon_0
-{\bf i} \varepsilon^{-1}\int_0^t (S_{t-r}^{\varepsilon, \alpha}-S_{t-r}^{\varepsilon, \alpha, N}) dW_r.
\end{align*} 
By the inequality \eqref{s-sn} and the condition \eqref{con-add}, 
\begin{align*}
\|e^N_t\|_{L_\omega^p L_x^2}
& \le \|(S_t^{\varepsilon, \alpha}-S_t^{\varepsilon, \alpha, N}) u^\varepsilon_0\|_{L_\omega^p L_x^2}
+\varepsilon^{-1} \|\int_0^t (S_{t-r}^{\varepsilon, \alpha}-S_{t-r}^{\varepsilon, \alpha, N}) dW_r\|_{L_\omega^p L_x^2} \\[1.5mm]
& \le e^{-\alpha t \varepsilon^{-1}} \lambda_N^{-\mu/2}\|u^\varepsilon_0\|_{L_\omega^p \dot H^\mu}
+ C \varepsilon^{-1} (\int_0^t \|S_r^{\varepsilon, \alpha}-S_r^{\varepsilon, \alpha, N}\|^2_{\LL_2^\mu} dr)^{1/2} \\[1.5mm]
& \le e^{-\alpha t \varepsilon^{-1}} \lambda_N^{-\mu/2}\|u^\varepsilon_0\|_{L_\omega^p \dot H^\mu}
+ C \varepsilon^{-1}  \lambda_N^{-\mu/2} (\int_0^t e^{-2 \alpha r/\varepsilon}dr)^{1/2},
\end{align*}
which shows \eqref{en+}. The completes the proof.
\end{proof}

\section{Spatio-temporal Full Discretization}
\label{sec5}

In this section, we investigate the full discretization of Eq. \eqref{semi-NLS}, spatially by the spectral Galerkin method in Section \ref{sec3} and temporally by the midpoint scheme. Let $T<M \in \nn_+$ and $\tau=T/M \in (0, 1)$ be the temporal step size.
Then, the midpoint-spectral Galerkin approximation of Eq. \eqref{semi-NLS} is to find 
\begin{equation}\label{full}
\begin{split}
u_{m+1}^{\varepsilon,N}=u_m^{\varepsilon,N}+&(\frac{{ \bf i}\varepsilon}{2}\Delta_N u_{m+1/2}^{\varepsilon,N}+{\bf i} \varepsilon^{-1}P_NF(u_{m+1/2}^{\varepsilon,N}))\tau
-{\bf i} \varepsilon^{-1}P_N G(u_m^{\varepsilon,N})\delta_m W,
\end{split}
\end{equation}
where $u_{m+1/2}:=(u_m+u_{m+1})/2$ and $\delta_m W:=W(t_{m+1})-W(t_m)$, $m \in \zz_M$. It is straightforward to see that
\begin{equation*}
\begin{split}
(I-\frac{{ \bf i}\varepsilon}{4}\Delta_N\tau)u_{m+1}^{\varepsilon,N}=&(I+\frac{{ \bf i}\varepsilon}{4}\Delta_N\tau)u_{m}^{\varepsilon,N}+{\bf i} \varepsilon^{-1}P_NF(u_{m+1/2}^{\varepsilon,N})\tau\\[1.5mm]
&-{\bf i} \varepsilon^{-1}P_N G(u_m^{\varepsilon,N})\delta_m W.
\end{split}
\end{equation*}
Denote by
$$
S_\tau:=S^{\varepsilon, N}_\tau:=(I-\frac{{ \bf i}\varepsilon}{4}\Delta_N\tau)^{-1}(I+\frac{{ \bf i}\varepsilon}{4}\Delta_N\tau),\quad 
T_\tau:=T^{\varepsilon, N}_\tau:=(I-\frac{{ \bf i}\varepsilon}{4}\Delta_N\tau)^{-1}.
$$
Then we have the full discrete mild formulation
\begin{equation} \label{mild-full}
\begin{split}
u_{m}^{\varepsilon,N}=S_\tau^mP_Nu_0^\varepsilon
&+ {\bf i} \varepsilon^{-1} \sum_{j=0}^{m-1} S_{\tau}^{m-j}T_\tau P_NF(u_{j+1/2}^{\varepsilon,N})\tau\\
&- {\bf i} \varepsilon^{-1} \sum_{j=0}^{m-1}  S_\tau^{m-j}T_\tau P_N G(u_j^{\varepsilon,N})\delta_j W.
\end{split}
\end{equation}

It is not difficult to show the well-posedness and moment's estimate of the midpoint-spectral Galerkin approximation \eqref{full} under Assumption \ref{ap-fg}.

\begin{lm}  \label{lm5.1}
Let $p \ge 2$, $\mu > 0$, $u^\varepsilon_0 \in L^p(\Omega; \dot H^\mu)$, and Assumption \ref{ap-fg} hold.
Then there exists $\tau_0 \in (0, 1)$ such that for any $\tau \in (0, \tau_0)$, the midpoint-spectral Galerkin approximation \eqref{full} is uniquely solved and there exists a positive constant $C$ such that 
\begin{align}\label{est-ujn}
\sup_{m \in \zz_M} \|u_m^{\varepsilon,N}\|_{L_\omega^p \dot H^\mu} 
\le C e^{CT \varepsilon^{-2}} (1+ \|u^\varepsilon_0\|_{L_\omega^p \dot H^\mu}).
\end{align}     
In particular, if $F(u)={ \bf i} \alpha u$ with $\alpha>0$ and $G={\rm Id}$ such that \eqref{con-add} holds, then for any $m \in \zz_M$, 
\begin{align} \label{est-ujn+} 
\|u_m^{\varepsilon,N}\|_{L_\omega^p \dot H^\mu}
& \le e^{-\alpha t_m \varepsilon^{-1}} \|u^\varepsilon_0\|_{L_\omega^p \dot H^\mu}
+C (\alpha\varepsilon)^{-1/2}.
\end{align}  
\qed
\end{lm}

The proof of Lemma \ref{lm5.1} is similar to the same as in the proof of Theorem \ref{tm-well} and thus is omitted here. Now, we are in the position to show the approximate error between $u^{\varepsilon, N}(t_m)$ and $u_{m}^{\varepsilon, N}$.
\begin{tm} \label{tm-tau}
Let $p \ge 2$, $\mu > 0$, $u^\varepsilon_0 \in L^p(\Omega; \dot H^\mu)$, and Assumption \ref{ap-fg} hold.
Then there exists a positive constant $C$ such that  
\begin{align} \label{tau}
&\sup_{m \in \zz_M} \|u^{\varepsilon,N}(t_m)-u_{m}^{\varepsilon,N}\|_{L_\omega^p L_x^2}
\le C e^{C T \varepsilon^{-2}} (1+\|u_0^\varepsilon\|_{L_\omega^p \dot H^\mu}) \tau^{(\mu/3-1/2) \wedge 1/2}.
\end{align} 
Moreover, if $F(u)={ \bf i} \alpha u$ with $\alpha>0$ and $G={\rm Id}$ such that \eqref{con-add} holds, then 
\begin{align} \label{tau+} 
\sup_{m \in \nn} \|u^{\varepsilon,N}(t_m)-u_{m}^{\varepsilon,N}\|_{L_\omega^p L_x^2}
\le C(\|u^\varepsilon_0\|_{L_\omega^p \dot H^\mu}
+(\alpha \varepsilon)^{-1/2}) \tau^{(\mu/3-1/2) \wedge 1/2}.
\end{align} 	
\end{tm}

\begin{proof}
Note that subtracting Eq. \eqref{mild-full} from Eq. \eqref{mild-N} with $t=t_m$ implies that
\begin{align*}
u^{\varepsilon,N}(t_m)-u_{m}^{\varepsilon,N}
&=(S_{t_m}-S_\tau^m) P_N  u_0^\varepsilon\\[1.5mm]
& \quad +{\bf i} \varepsilon^{-1}\sum_{j=0}^{m-1} \int_{t_j}^{t_{j+1}} [S_{t_m-r} P_N F(u_r)-S_{\tau}^{m-j}T_\tau P_N F(u_{j+1/2}^{\varepsilon,N})] dr\\[1.5mm]
& \quad -{\bf i} \varepsilon^{-1}\sum_{j=0}^{m-1} \int_{t_j}^{t_{j+1}}[S_{t_m-r} P_N G(u_r)-S_{\tau}^{m-j}T_\tau P_N G(u_j^{\varepsilon,N})] dW_r \\[1.5mm]
&:=I_\tau+II_\tau+III_\tau. 
\end{align*}
 
 For the first term $I_\tau$, we use the inequality \eqref{s-stau} in Appendix to derive
 \begin{align} \label{1tau}
\|I_\tau\|_{L_\omega^p L_x^2}
& \le C t_m^{\mu/6} \varepsilon^{\mu/2} \tau^{\mu/3} \|u_0^\varepsilon\|_{L_\omega^p \dot H^\mu}.
 \end{align}
 For the second term $II_\tau$, using the Minkovski inequality, the inequality \eqref{s-stau1} with $\alpha=0$, and the condition \eqref{ap-f}, we obtain
 \begin{align*}
\|II_\tau\|_{L_\omega^p L_x^2}
 & \le \varepsilon^{-1} \sum_{j=0}^{m-1} \int_{t_j}^{t_{j+1}} \|(S_{t_m-r}-S_{\tau}^{m-j}T_\tau) P_N F(u_{j+1/2}^{\varepsilon,N})\|_{L_\omega^p L_x^2}  dr\\[1.5mm]
 & \quad +\varepsilon^{-1} \sum_{j=0}^{m-1} \int_{t_j}^{t_{j+1}}\|(S_{t_m-r} P_N  [F(u_r)-F(u_{j+1/2}^{\varepsilon,N})] \|_{L_\omega^p L_x^2}  dr  \\[1.5mm]   
& \le \varepsilon^{-1} \sum_{j=0}^{m-1} \int_{t_j}^{t_{j+1}}  \|(S_{t_m-r}-S_\tau^{m-j} T_\tau)P_N\|_{\LL(\dot H^\mu, L^2)}
\|F(u_{j+1/2}^{\varepsilon,N})\|_{L_\omega^p \dot H^\mu}dr\\[1.5mm] 
 & \quad +\varepsilon^{-1} \sum_{j=0}^{m-1} \int_{t_j}^{t_{j+1}}\|F(u_r)-F(u_{j+1/2}^{\varepsilon,N}) \|_{L_\omega^p L_x^2}  dr  \\[1.5mm]
 & \le C t_m^\beta (1+\|u_j^{\varepsilon,N}\|_{L_\omega^p \dot H^\mu}) \tau^{(\mu/3) \wedge 1}
 +\varepsilon^{-1} \sum_{j=0}^{m-1} \int_{t_j}^{t_{j+1}}\|u_r-u_{j+1/2}^{\varepsilon,N} \|_{L_\omega^p L_x^2} dr.
 \end{align*} 
By the H\"older continuous of $u_t^\varepsilon$ in \eqref{hol}, we get 
 \begin{align*}
& \varepsilon^{-1} \sum_{j=0}^{m-1} \int_{t_j}^{t_{j+1}}
\|u_r-u_{j+1/2}^{\varepsilon,N}\|_{L_\omega^p L_x^2} dr\\[1.5mm] 
& \le C \varepsilon^{-1} \sum_{j=0}^{m-1} \int_{t_j}^{t_{j+1}}
\|u_r-(u_{t_j}+u_{t_{j+1}})/2\|_{L_\omega^p L_x^2} dr\\[1.5mm] 
& \quad + C \varepsilon^{-1} \sum_{j=0}^{m-1}  \|\frac{u_{t_j}-u_j^{\varepsilon,N}}2+\frac{u_{t_{j+1}}-u_{j+1}^{\varepsilon,N}}2\|_{L_\omega^p L_x^2} \tau \\[1.5mm] 
& \le C \varepsilon^{-1} \sum_{j=0}^{m-1} \int_{t_j}^{t_{j+1}}
\|u_r-(u_{t_j}+u_{t_{j+1}})/2\|_{L_\omega^p L_x^2} dr\\[1.5mm] 
& \quad + C \varepsilon^{-1} \sum_{j=0}^{m-1}  \|u_{t_j}-u_j^{\varepsilon,N}\|_{L_\omega^p L_x^2} \tau 
+ C \varepsilon^{-1} \sum_{j=0}^{m-1}  \|u_{t_{j+1}}-u_{j+1}^{\varepsilon,N}\|_{L_\omega^p L_x^2} \tau \\[1.5mm] 
& \le C\varepsilon^{-1}  e^{C t_m \varepsilon^{-2}} (1+\|u_0^\varepsilon\|_{L_\omega^p \dot H^\mu})\tau^{1/2}   \\[1.5mm] 
& \quad + C \varepsilon^{-1} \sum_{j=0}^{m-1}  \|u_{t_j}-u_j^{\varepsilon,N}\|_{L_\omega^p L_x^2} \tau 
+ C \varepsilon^{-1} \|u_{t_m}-u_m^{\varepsilon,N}\|_{L_\omega^p L_x^2} \tau,
\end{align*} 
which by combining with \eqref{est-ujn}, yields that  
\begin{equation} \label{2tau}
\begin{split}
\|II_\tau\|_{L_\omega^p L_x^2} 
& \le C e^{C t_m \varepsilon^{-2}} (1+\|u_0^\varepsilon\|_{L_\omega^p \dot H^\mu})\tau^{1/2}  \\[1.5mm] 
& \quad + C \varepsilon^{-1} \sum_{j=0}^{m-1}  \|u_{t_j}-u_j^{\varepsilon,N}\|_{L_\omega^p L_x^2} \tau 
+ C \varepsilon^{-1} \|u_{t_m}-u_m^{\varepsilon,N}\|_{L_\omega^p L_x^2} \tau. 
\end{split}
\end{equation}
 
Similarly, using the Minkovski inequality, the discrete Burkholder--Davis--Gundy inequality, the inequality \eqref{s-stau1} with $\alpha=0$, and the condition \eqref{ap-g}, we have
  \begin{equation}  \label{3tau}
  \begin{split}
&\|III_\tau\|_{L_\omega^p L_x^2}\\[1.5mm]
 & \le \varepsilon^{-1} (\sum_{j=0}^{m-1} \int_{t_j}^{t_{j+1}}\|(S_{t_m-r}-S_{\tau}^{m-j}T_\tau) P_N  G(u_j^{\varepsilon,N})\|^2_{L_\omega^p \LL_2^0}  dr)^{\frac12}  \\[1.5mm]
  & \quad +\varepsilon^{-1} (\sum_{j=0}^{m-1} \int_{t_j}^{t_{j+1}}\|S_{t_m-r} P_N  [G(u_r)- G(u_j^{\varepsilon,N})]\|^2_{L_\omega^p \LL_2^0}  dr)^{\frac12} \\[1.5mm]
 & \le \varepsilon^{-1} \sum_{j=0}^{m-1} (\int_{t_j}^{t_{j+1}}  \|(S_{t_m-r}-S_\tau^{m-j} T_\tau)P_N\|^2_{\LL(\dot H^\mu, L^2)}
\|G(u_{j+1/2}^{\varepsilon,N})\|^2_{L_\omega^p \LL_2^\mu}dr)^{1/2}  \\[1.5mm] 
 & \quad +\varepsilon^{-1}  (\sum_{j=0}^{m-1} \int_{t_j}^{t_{j+1}}\|G(u_r)-G(u_{j+1/2}^{\varepsilon,N}) \|^2_{L_\omega^p \LL_2^0}  dr)^{1/2}  \\[1.5mm]
 & \le C e^{C t_m \varepsilon^{-2}} (1+\|u_0^\varepsilon\|_{L_\omega^p \LL_2^\mu}) \tau^{(\mu/3-1/2) \wedge 1/2}    \\[1.5mm]
& \quad + C \varepsilon^{-1} \sum_{j=0}^{m-1}  \|u_{t_j}-u_j^{\varepsilon,N}\|_{L_\omega^p L_x^2} \tau 
+ C \varepsilon^{-1} \|u_{t_m}-u_m^{\varepsilon,N}\|_{L_\omega^p L_x^2} \tau.
\end{split}
 \end{equation} 

Combing \eqref{1tau}, \eqref{2tau} and \eqref{3tau}, we find
 \begin{align*} 
\|u^{\varepsilon,N}(t_m)-u_{m}^{\varepsilon,N}\|_{L_\omega^p L_x^2}
& \le C e^{C t_m \varepsilon^{-2}} (1+\|u_0^\varepsilon\|_{L_\omega^p \dot H^\mu})\tau^{(\mu/3-1/2) \wedge 1/2}  \\[1.5mm] 
& \quad + C \varepsilon^{-1} \sum_{j=0}^{m-1}  \|u_{t_j}-u_j^{\varepsilon,N}\|_{L_\omega^p L_x^2} \tau   \nonumber \\[1.5mm]
& \quad + C \varepsilon^{-1} (\sum_{j=0}^{m-1}  
\|u_{t_j}- u_j^{\varepsilon,N}\|^2_{L_\omega^p L_x^2} \tau )^{\frac12} \\[1.5mm]
& \quad + C \varepsilon^{-1} \|u_{t_m}-u_m^{\varepsilon,N}\|_{L_\omega^p L_x^2} \tau. 
 \end{align*}
 Therefore, for $\tau<C \varepsilon^{-1}$ which always holds for sufficiently small $\varepsilon$, we get 
 \begin{align*} 
\|u^{\varepsilon,N}(t_m)-u_{m}^{\varepsilon,N}\|^2_{L_\omega^p L_x^2}
& \le C e^{C t_m \varepsilon^{-2}} (1+\|u_0^\varepsilon\|_{L_\omega^p \dot H^\mu})^2 \tau^{2[(\mu/3-1/2) \wedge 1/2]}  \\[1.5mm]  
& \quad + C \varepsilon^{-1} \sum_{j=0}^{m-1}  
\|u_{t_j}- u_j^{\varepsilon,N}\|^2_{L_\omega^p L_x^2} \tau.
 \end{align*}
 The desired result \eqref{tau} is then obtained by the Gronwall inequality. 

Finally, to show \eqref{tau+}, it suffices to use the estimate \eqref{s-stau}, \eqref{s-stau1}, and \eqref{s-stau2} with $\alpha>0$, in addition to the previous arguments.
For instance, let us consider the stochastic convolution in the additive noise case where we apply \eqref{s-stau1} with $\alpha>0$: 
 \begin{align*}
\|III_\tau\|_{L_\omega^p L_x^2} 
& \le \varepsilon^{-1} (\sum_{j=0}^{m-1} \int_{t_j}^{t_{j+1}}\|S_{t_m-r}-S_{\tau}^{m-j}T_\tau\|^2_{L_\omega^p \LL_2^0}dr)^{\frac12}  \\[1.5mm]
& \le C e^{-\alpha t_m \varepsilon^{-1}} \tau^{(\mu/3-1/2) \wedge 1/2}.
 \end{align*} 
The situation is similar to the estimates of the other terms.  This finishes the proof of \eqref{tau+}.
\end{proof}

Combining Theorems \ref{tm-en} and \ref{tm-tau} will then give the following strong convergence rate between the exact solution of Eq. \eqref{semi-NLS} and its full discretization \eqref{full}.
\begin{tm}
Let $p \ge 2$, $\mu \ge 0$, $u^\varepsilon_0 \in L^p(\Omega; \dot H^\mu)$, and Assumption \ref{ap-fg} hold.
Then  
\begin{align} \label{entau}
&\sup_{m \in \zz_M} \|u^{\varepsilon}(t_m)-u_{m}^{\varepsilon,N}\|_{L_\omega^p L_x^2}
\le C e^{C T \varepsilon^{-2}} (1+\|u_0^\varepsilon\|_{L_\omega^p \dot H^\mu}) (N^{-\mu/d}+\tau^{(\mu/3-1/2) \wedge 1/2}).
\end{align} 
If $F(u)={ \bf i} \alpha u$ with $\alpha>0$ and $G={\rm Id}$ such that \eqref{con-add} holds, then there exists a positive constant $C$ such that   
\begin{align} \label{entau+} 
\sup_{m \in \nn} \|u^{\varepsilon}(t_m)-u_{m}^{\varepsilon,N}\|_{L_\omega^p L_x^2}
\le C (\|u^\varepsilon_0\|_{L_\omega^p \dot H^\mu}
+ (\alpha \varepsilon)^{-1/2}) (N^{-\mu/d}+\tau^{(\mu/3-1/2) \wedge 1/2}).
\end{align} 	
\end{tm}

\begin{rk} 
We can formulate the following meshing strategy based on \eqref{entau} and \eqref{entau+}. 
For instance, let $F(u)={ \bf i} \alpha u$ with $\alpha>0$, $G={\rm Id}$ such that \eqref{con-add} holds, and $\delta>0$ be the desired error bound. Then
$\sup_{m \in \nn} \|u^{\varepsilon}(t_m)-u_{m}^{\varepsilon,N}\|_{L_\omega^p L_x^2} \le \delta$ holds provided $\tau=N^{-\frac{2\mu}d}$ with $\mu \ge 3$ or $\tau=N^{-\frac{6\mu}{d(2s-3)}}$ with $\mu \in [2, 3)$ and 
\begin{align*}
N^{-\frac{2\mu}d}/\varepsilon \le C \delta^2.
\end{align*} 
\end{rk}

\section*{Appendix}

In this appendix, we collect several regularity results and error estimates of the Schr\"odinger semigroup and its discretizations.

\begin{lm}
Let $\mu \ge 0, \rho \in [0, 1]$, $u \in \dot H^\mu$, $v \in \dot H^\rho$, $\alpha \in \rr$, and $t \ge 0$.
Then  
\begin{align}  
\|S_t^{\alpha,\varepsilon} u\|_\mu & = e^{-\alpha t \varepsilon^{-1}} \|u\|_\mu, \label{s} \\[1.5mm]
\|(S^\varepsilon_t-{\rm Id}) v\| & \le C \varepsilon^{\rho} t^{\rho} \|v\|_{2\rho}, \label{s-con} \\[1.5mm]
\|(S^{\alpha,\varepsilon}_t-{\rm Id}) v\| & \le C \varepsilon^{\rho} t^{\rho} e^{-\alpha t \varepsilon^{-1}} \|v\|_{2\rho}+\|v\|^2. \label{s-con+}
\end{align} 
\end{lm}

\begin{proof}
By the definitions of $S_t^{\alpha,\varepsilon}$ and $\|\cdot\|_\mu$-norm, in combination with the increasing property of $\lambda_k$ with respect to $k \in \nn_+$, we have 
\begin{align*} 
\|S_t^{\alpha,\varepsilon}v\|_\mu^2
& =\sum_{k=1}^\infty \lambda_k^\mu \langle e_k, v\rangle^2|e^{(\frac{{ \bf i}\varepsilon}{2} \lambda_k-\frac\alpha \varepsilon) t}|^2
=e^{-\frac{2\alpha}\varepsilon t} \sum_{k=1}^\infty \lambda_k^\mu \langle e_k, v\rangle^2 
= e^{-\frac{2\alpha}\varepsilon t} \|v\|_\mu^2,
\end{align*}
which shows \eqref{s}.

Similarly,
\begin{align*} 
\|(S^{\alpha,\varepsilon}_t-{\rm Id}) v\|^2 
& =\sum_{k=1}^\infty \langle e_k, v\rangle^2|e^{(\frac{{ \bf i}\varepsilon}{2} \lambda_k-\frac\alpha \varepsilon) t}-1|^2.
\end{align*} 
It is clear that for any $\beta \in [0, 1]$, 
\begin{align} \label{exp}
|e^{{\bf i} x}-e^{{\bf i} y}| =2 |\sin[ (x-y)/2]| \le 2 |x-y|^\beta,
\quad \forall~ x, y \in \rr.
\end{align}
Then 
\begin{align*} 
\|(S^\varepsilon_t-{\rm Id}) v\|^2 
& \le \left|\frac{\varepsilon}{2} t\right|^{2\rho} 
\sum_{k=1}^\infty|\lambda_k|^{2\rho} \langle e_k, v\rangle^2 
\le C \varepsilon^{2 \rho} t^{2 \rho}  \|v\|_{2\rho}^2,
\end{align*}
which shows \eqref{s-con}.
To show \eqref{s-con+}, we just need to note that 
\begin{align*} 
|e^{(\frac{{ \bf i}\varepsilon}{2} \lambda_k-\frac\alpha \varepsilon) t}-1|
\le e^{-\frac{\alpha}\varepsilon t} |e^{\frac{{ \bf i}\varepsilon}{2} \lambda_kt}-1|
+|1-e^{-\frac{\alpha}\varepsilon t}|
\le 2e^{-\frac{\alpha}\varepsilon t} \left|\frac{\varepsilon}{2} t \lambda_k\right|^{\rho}+1,
\end{align*}
and then use the previous argument.
\end{proof}

\begin{lm}
Let $\mu \ge 0$, $v \in \dot H^\mu$, $\alpha \in \rr$, and $t \ge 0$.
Then  
\begin{align}  
&\|(S_t^{\alpha,\varepsilon}-S_t^{\alpha,\varepsilon, N}) v\| 
\le e^{-\alpha t \varepsilon^{-1}} \lambda_N^{-\mu/2}\|v\|_\mu. \label{s-sn}
\end{align} 
\end{lm}

\begin{proof}
By the definitions of $S_t^{\alpha,\varepsilon}$, $S_t^{\alpha,\varepsilon, N}$, and $\|\cdot\|_\mu$-norm, in combination with the increasing property of $\lambda_k$ with respect to $k \in \nn_+$, we have 
\begin{align*} 
\|(S_t^{\alpha,\varepsilon}-S_t^{\alpha,\varepsilon, N}) v\|^2
& =\sum_{k=N+1}^\infty\langle e_k, v\rangle^2|e^{(\frac{{ \bf i}\varepsilon}{2} \lambda_k-\frac\alpha \varepsilon) t}|^2
=e^{-\frac{2\alpha}\varepsilon t} \sum_{k=N+1}^\infty \langle e_k, v\rangle^2  \\
&\le e^{-\frac{2\alpha}\varepsilon t} \lambda_N^{-\mu}\sum_{k=N+1}^{\infty}|\lambda_k|^\mu\langle e_k, v\rangle^2
\le e^{-\frac{2\alpha}\varepsilon t} \lambda_N^{-\mu} \|v\|_\mu^2, 
\end{align*}
which implies the assertion.
\end{proof}

\begin{lm}\label{lm5.4}
Let $\mu \in (0, 6]$, $\alpha \in \rr$, and $k \in \nn_+$.
Then  
\begin{align}   \label{s-stau}
&\|(S^{\alpha,\varepsilon}(t_k)-(S_\tau^{\alpha,\varepsilon})^k) \|_{\LL(\dot H^{\mu}, L^2)} 
\le C e^{-\alpha t_k/\varepsilon} t_k^{\mu/6} \varepsilon^{\mu/2} \tau^{\mu/3}. 
\end{align} 
\end{lm}

\begin{proof}
By the definitions of $S_t^{\alpha,\varepsilon}$, $S_t^{\alpha,\varepsilon, N}$, and $\|\cdot\|_\mu$-norm, in combination with the increasing property of $\lambda_k$ with respect to $k \in \nn_+$, we have 
\begin{align*} 
& \|(S^{\alpha,\varepsilon}(t_k)-(S_\tau^{\alpha,\varepsilon})^k) \|_{\LL(\dot H^\mu, L^2)}^2 \\[1.5mm]
& = \sup_{\|v\|_\mu=1} \|(S^{\alpha,\varepsilon}(t_k)-(S_\tau^{\alpha,\varepsilon})^k) v\|^2 \\[1.5mm]
& = \sup_{\|v\|_\mu=1} \sum_{m=1}^\infty \<(S^{\alpha,\varepsilon}(t_k)-(S_\tau^{\alpha,\varepsilon})^k) v, e_m\>^2 \\[1.5mm]
& =e^{-2\alpha t_k/\varepsilon} \sup_{\|v\|_\mu=1} \sum_{m=1}^\infty |e^{\frac{{\bf i} \varepsilon t_k \lambda_m}2} - (\frac{1+{\bf i} \varepsilon \tau \lambda_m/4}{1-{\bf i} \varepsilon \tau \lambda_m/4})^k |^2 \<v,  e_m\>^2 \\[1.5mm]
& \le e^{-2\alpha t_k/\varepsilon} (\sup_{m \in \nn_+} |e^{\frac{{\bf i} \varepsilon t_k \lambda_m}2} - (\frac{1+{\bf i} \varepsilon \tau \lambda_m/4}{1-{\bf i} \varepsilon \tau \lambda_m/4})^k |^2 \lambda_m^{-\mu})
(\sup_{\|v\|_\mu=1} \sum_{k=1}^\infty \lambda_k^\mu \<v,  e_k\>^2 ) \\[1.5mm]
& =e^{-2\alpha t_k/\varepsilon} \sup_{m \in \nn_+}\left|e^{{\bf i} \varepsilon t_k \lambda_m/2} - e^{2 {\bf i} k \arctan(\varepsilon \tau \lambda_m/4)}\right|^2 \lambda_m^{-\mu}.
\end{align*}
Using the inequality \eqref{exp}, we have   
\begin{equation*}
\begin{split} 
&\|(S^{\alpha,\varepsilon}(t_k)-(S_\tau^{\alpha,\varepsilon})^k) \|_{\LL(\dot H^\mu, L^2)}^2\\[1.5mm]
&\le e^{-2\alpha t_k/\varepsilon} \sup_{k \in \nn_+} (2k)^{2 \beta} |\varepsilon \tau \lambda_k/4- \arctan(\varepsilon \tau \lambda_k/4)|^{2 \beta} \lambda_k^{-\mu}.
\end{split}
\end{equation*}
By the elementary inequality $|x-\arctan(x)| \le C |x|^3$, $x \in \rr$, we get 
\begin{align*} 
\|(S^{\alpha,\varepsilon}(t_k)-(S_\tau^{\alpha,\varepsilon})^k) \|_{\LL(\dot H^\mu, L^2)} 
& \le C e^{-\alpha t_k/\varepsilon} t_k^\beta \tau^{2 \beta} \varepsilon^{3 \beta}  \sup_{k \in \nn_+}\lambda_k^{-(\mu/2-3 \beta)},
\end{align*}
which shows \eqref{s-stau} with $\beta=\mu/6$.  
\end{proof}

\begin{lm}
Let $\beta \in [1/3, 1]$, $j \in \nn_+$, and $r \in [t_{j-1}, t_j]$.
Then  
\begin{align}   \label{s-stau-}
& \|S_{-r}-S_{\tau}^{-j}T_\tau\|_{\mathcal L(\dot H^{6 \beta}; L^2)}
\le C t_j^\beta \tau^{(2 \beta) \wedge 1} \varepsilon. 
\end{align} 
Let $s \ge 2$, $u \in L^p(\Omega; H^1)$, $\alpha \in \rr$, and $t \ge 0$.
Then  
\begin{align}   
& \int_{t_j}^{t_{j+1}} \|(S_{t_m-r}-S_{\tau}^{m-j}T_\tau) P_N\|_{\mathcal L(\dot H^\mu; L^2)}  dr   
\le C e^{-\alpha t_m \varepsilon^{-1}} t_m^{s/6} \varepsilon \tau^{1+(\mu/3) \wedge 1}, \label{s-stau1}\\
& (\int_{t_j}^{t_{j+1}} \|(S_{t_m-r}-S_{\tau}^{m-j}T_\tau) P_N\|^2_{\mathcal L(\dot H^\mu; L^2)} dr)^\frac12    
\le C e^{-\alpha t_m \varepsilon^{-1}} t_m^{s/6} \varepsilon \tau^{1/2+(\mu/3) \wedge 1}.\label{s-stau2}
\end{align} 
\end{lm}

\begin{proof} 
Since $S_\tau$ is isometric in $L^2$, we have 
\begin{align*}   
& \|S_{-r}-S_{\tau}^{-j}T_\tau\|_{\mathcal L(\dot H^\mu; L^2)} \\[1.5mm]
& \le \|S_{-r}-S_{-t_j}\|_{\mathcal L(\dot H^\mu; L^2)}
+ \|S_{-t_j}-S_{\tau}^{-j}\|_{\mathcal L(\dot H^\mu; L^2)}
+ \|S_{\tau}^{-j}({\rm Id}-T_\tau)\|_{\mathcal L(\dot H^\mu; L^2)} \\[1.5mm]
& = \|S_{-r}-S_{-t_j}\|_{\mathcal L(\dot H^\mu; L^2)}
+ \|S_{-t_j}-S_{\tau}^{-j}\|_{\mathcal L(\dot H^\mu; L^2)}
+ \|{\rm Id}-T_\tau\|_{\mathcal L(\dot H^\mu; L^2)}.
\end{align*}
As in the proof of Lemma \ref{lm5.4}, it holds that
\begin{align*}  
& \|S_{-r}-S_{-t_j}\|_{\mathcal L(\dot H^\mu; L^2)} \\[1.5mm]
& \le \sup_{k \in \nn_+} |e^{{\bf i} \varepsilon r \lambda_k/2} - e^{2 {\bf i} j \arctan(\varepsilon \tau \lambda_k/4)}| \lambda_k^{-\mu/2} \\[1.5mm]
& \le \sup_{k \in \nn_+} |e^{{\bf i} \varepsilon r \lambda_k/2} - e^{{\bf i} \varepsilon t_j \lambda_k/2}| \lambda_k^{-\mu/2}
+ \sup_{k \in \nn_+} |e^{{\bf i} \varepsilon t_j \lambda_k/2} - e^{2 {\bf i} j \arctan(\varepsilon \tau \lambda_k/4)}| \lambda_k^{-\mu/2} \\[1.5mm] 
& \le \sup_{k \in \nn_+} |\varepsilon \lambda_k (r-t_j)| \lambda_k^{-\mu/2}
+\sup_{k \in \nn_+} (2j)^\beta |\varepsilon \tau \lambda_k/4 - \arctan(\varepsilon \tau \lambda_k/4)|^\beta \lambda_k^{-\mu/2} \\[1.5mm]
& \le \tau \varepsilon \sup_{k \in \nn_+} \lambda_k^{-(\mu/2-1)}
+ C t_j^\beta \tau^{2 \beta} \varepsilon^{3 \beta}.
\end{align*}
Similarly,
\begin{align*}  
\|S_{-t_j}-S_\tau^{-j}\|_{\mathcal L(\dot H^\mu; L^2)} 
\le C t_j^\beta \tau^{2 \beta} \varepsilon^{3 \beta},
\end{align*}
and 
\begin{align*}  
\|{\rm Id}-T_\tau\|_{\mathcal L(\dot H^\mu; L^2)} 
& \le \sup_{k \in \nn_+}|1 - (1-{\bf i} \varepsilon \tau \lambda_k/4)^{-1}|  \lambda_k^{-\mu/2} 
\le \frac14 \tau \varepsilon \sup_{k \in \nn_+} \lambda_k^{-(\mu/2-1)}.
\end{align*}
Combining the above estimates, equality can be derived from the equality \eqref{s-stau-}.

Finally, as $S_t$, $t \in \rr$, is isometric in $L^2$, we have 
\begin{align*}   
\|(S_{t_m-r}^\alpha-S_{\tau}^{\alpha, m-j}T_\tau) P_N \|_{\mathcal L(\dot H^\mu; L^2)}
&  =e^{-\alpha t_m \varepsilon^{-1}}\|S_{-r}-S_{\tau}^{-j}T_\tau\|_{\mathcal L(\dot H^\mu; L^2)} \\[1.5mm]
& \le C e^{-\alpha t_m \varepsilon^{-1}} t_j^\beta \tau^{(2 \beta) \wedge 1} \varepsilon.
\end{align*}
which shows \eqref{s-stau1} and \eqref{s-stau2}.

\end{proof}



\bibliographystyle{amsalpha}
\bibliography{semi_classical} 
 
\end{document}